\documentclass{article} 
\usepackage{nips}
\usepackage{hyperref}
\usepackage{url}
\usepackage{amsmath,amsthm}
\usepackage{amssymb}
\usepackage{chicago}

\newcommand{\ls}{\left[}
\newcommand{\rs}{\right]}
\newcommand{\lf}{\left\{}
\newcommand{\la}{\left\langle}
\newcommand{\ra}{\right\rangle}
\newcommand{\rf}{\right\}}

\newcommand{\R}{\mathbb{R}}
\newcommand{\cW}{\mathcal{W}}
\newcommand{\cX}{\mathcal{X}}
\newcommand{\cY}{\mathcal{Y}}
\newcommand{\cF}{\mathcal{F}}

\newcommand{\eps}{\epsilon}
\newcommand{\nom}{\nabla\psi}

\renewcommand{\P}{\mathbb{P}}
\renewcommand{\Pr}{\mathbb{P}}
\newcommand{\E}{\mathbb{E}}
\newcommand{\lv}{\left\vert}
\newcommand{\rv}{\right\vert}

\newcommand\argmin{\mylim{arg\,min\,}}
\newcommand\argmax{\mylim{arg\,max\,}}

\title{ MМulti-Class Classification: \\Mirror Descent Approach}

\author{
Daria~Reshetova
\\
Moscow Institute of Physics and Technology\\
Institute for Information Transmission Problems\\
\texttt{reshetova@phystech.edu} \\
}

\newtheorem{lemma}{Lemma}
\newtheorem{cor}{Corollary}

\nipsfinalcopy 

\begin{document}

\maketitle

\begin{abstract}
We consider the problem of multi-class classification and a stochastic optimization approach to it. We derive risk bounds for stochastic mirror descent algorithm and provide examples of set geometries that make the use of the algorithm efficient in terms of error in the number of classes.
\end{abstract}

\section{Introduction}

Classification is one of the core machine learning tasks. Multi-class classification arises in various problems including document classification \cite{rennie2001improving}, image \cite{foody2004relative,lee1997unsupervised}, gesture \cite{mcneill1992hand} and video recognition \cite{karpathy2014large} and many others. Datasets for the problems are growing in both number of samples and number of classes $k$. As the expected error of classification algorithms also increases, the growth rate in terms of $k$ becomes crucial.

In this paper we consider the classical one-vs-all margin classification approach \cite{aly2005survey}, which was empirically shown to be as good as ECOC and all-vs-all approach, at least from the practical point of view \cite{rifkin2004defense}. There are several generalization ability guarantees known for the class of learners. Distribution-independent ones rely on function class complexity measures such as  Natarajan 
 \cite{natarajan1989learning}, graph 
 \cite{natarajan1989learning,dudley2010universal} and Vapnic-Chervoninkis 
 \cite{guermeur2007vc} dimension. If we consider kernel separators $f_y(x, w)=K(x, w_y)$ with PSD K, the bounds lead to $\tilde{O}(k/\sqrt{n})$ excess risk \cite{daniely2013multiclass,guermeur2007vc}, Covering number based bound presented in \cite{zhang2002covering} gives $\tilde{O}(\sqrt{k/n})$ rate. While the bounds provided seem tight, they result in large constants and dimension dependence for combinatorial complexity measures. As for distribution-dependent bounds, little is known beyond the general Rademacher complexity based bound, which is in the worst case of order $O(k/\sqrt{n})$  for typical function classes (e. g. finite VC-dimensional). 

As the underlying distribution of the pairs object-class is unknown, the problem can be solved by substituting the actual risk by the empirical one or by means of stochastic optimization. While the general problem statement of minimizing the risk allows the implementation of different optimization algorithms, first order methods are preferable to high-order ones for large-scale problems in terms of their generalization ability and computational efficiency \cite{bousquet2008tradeoffs}. We consider an adaptation of stochastic Mirror Descent algorithm \cite{nemirovsky1982problem,beck2003mirror} to solve the problem. Mirror Descent is a first-order algorithm for convex function minimization, which restricts the method to convex $\Omega$ and convex in $w$ loss-function $\ell(x,y,w),$ allowing dimension-independent excess risk bounds. 


The rest of the paper is organized as follows. In section 2 we describe the general assumptions made concerning the input space. In the next section we describe the algorithm and provide upper bounds for the expected error of the classifier. In section 4 we provide the bounds for large deviations probabilities and the examples of parameter set  geometries that allow faster rates.

\section{Preliminaries}
We first describe the framework of multi-class classification used in the paper.
Let $\cX\subseteq \R^d$ be a set of instances, $\cY=\{1\ldots k\}$ be a set of classes. The general assumption is that $\cX\times\cY$  support a probability space $(\cX\times \cY, A, \P)$ and the sample $S=\{x_i, y_i\}_{i=1}^n$ is i.i.d. drawn from the distribution. We denote $\mathcal{F}=\{f(\cdot, \cdot, w):\cX\times \cY\to \R\vert w\in\cW\}$ the class of decision functions -- a parametric class of measurable functions and $\ell: \cX\times \cY\times \cW \to \R_{+}$ the loss function.
We consider one-vs-all approach to the problem by setting the predictor to be 
$\hat{f}(x, w^*)=\underset{y\in \cY}{\max}f(x,y, w^*).$ The loss function is chosen to be a Lipschitz upper bound on the indicator function 
$[\hat{f}(x_i,w)=y_i]$ and to maximize the margin:
\[m(x, y, w)=f(x,y,w)-\underset{\hat{y}\in\cY,\, \hat{y}\neq y}
{\max}f(x,\hat{y},w)\]
We use $\ell(x,y,w)=\max\{0, 1-m(x,y,w)/\rho\}$ as the loss finction. The problem is then to minimize the expected risk
 $F(w)=\E_{(x,y)\sim \P}\ell(x, y, w):$
\begin{equation}\label{prob1}
w^*=\underset{w\in\cW}\argmin \E_{(x, y)}\ell(m(x,y,w))
\end{equation}

To make the application of mirror descent possible, the underlying function has 
to be convex, the fact that the distribution $\Pr$ is unknown leads to $f(x, y, w)$ needing to be linear in $w$, so 
$\cF=\{\la x,w_y\ra\,\vert\, w\in\cW\}$ further in the paper. The case can also be generalized to PSD-kernel classification via linear classifiers in RKHS \cite{mohri2012foundations}.

\section{Oracle inequalities}
Mirror descent algorithm is similar to stochastic gradient descent, except that for $\cW\subset E$ with $E$ being a Euclidean space it ensures gradient steps to be made in $E^*$ by mapping there with $\nabla \psi,$ where $\psi: E\to \R$ is a
strongly convex function with gradient field continuous on $\cW$:
\[\psi(w^1)-\psi(w^2)-\la \nabla\psi(w^2), w^1-w^2\ra
\geq \frac{1}{2} \|w^1-w^2\|^2.\]
As the number of classes is a factor of the dimension of $\cW,$ choosing the right proximity to measure the set diameter can effectively lower the error rate.

Mirror Descent steps are gradient steps with Bregman divergence of $\psi$ in the role of the distance:
\[w^1=\underset{w\in\cW}{\argmin}\psi(w)\]
$$w^{m+1}=\underset{w\in W}{\argmin} \{\Delta (w, w^m)+\alpha_m\langle F'(w^m), w-w^m\rangle\},$$
$\Delta(w^1, w^2)=\psi(w^1)-\psi(w^2)-\la \nabla\psi(w^2), w^1-w^2\ra \geq \frac{1}{2} \|w^1-w^2\|^2.$ In case of expectation minimization gradients are taken at random points
$g^k\in \partial \ell(x_k, y_k, w^k)$, which ensures  $\E g^k\in \partial F(w^k)$
as long as $(x_k, y_k)$ and $w^k$ are independent. 
\begin{lemma} \cite{nemirovski2004prox,beck2003mirror}\label{step_ineq}
For all $w\in \cW$ 
$$\Delta(w, w^{m+1}) \leq \alpha_m\la g^m, w-w^{m+1}\ra+\Delta(w, w^m)-\Delta(w^{m+1}, w^m) $$
\end{lemma}
\begin{proof}
Set $h(w)=\Delta (w, w^m)+\alpha_m\langle g^m, w-w^m\rangle,$
 then $w^{m+1}=\underset{w\in\cW}\argmin\, h(w).$ \\
Optimality of
$w^{m+1}$ leads to $\la h'(w^{m+1}), w-w^{m+1}\ra\geq 0$

As long as 
$h'(w^{m+1})=\nabla\psi(w^{m+1})-\nabla\psi(w^m)+\alpha_m g^m,$
we can rearrange the terms and get
\begin{eqnarray*}
0 & \leq &\la\alpha_m g^m+\nom(w^{m+1})-\nom(w^m), w-w^{m+1}\ra\\
&=&\la\alpha_m g^m, w-w^{m+1}\ra-\la\nom(w^m), w-w^m\ra\\
& + & \la \nom(w^{m+1}), w-w^{m+1}\ra +\la\nom(w^m),w^{m+1}-w^m\ra\\
& = & \la\alpha_m g^m, w-w^{m+1}\ra - \Delta(w, w^{m+1})+\Delta(w, w^m)-\Delta(w^{m+1}, w^m)
\end{eqnarray*}
\end{proof}

Lemma \ref{step_ineq} and the fact that $\E g_k\in \partial F(w^k)$ result in an oracle inequality for stochastic Mirror Descent.
\begin{cor}
For $U^2=\underset{u,w\in\cW}{\argmax} \left(\psi(u)-\psi(w)\right),$ $G^2=\underset{w\in\cW}{\argmax}\E\|g(x,y,w)\|^2,$ with $g(x,y,w)\in\partial \ell(x,y,w)$ and for any $w\in \cW$ and 
$w^{(n)}=\frac{\sum_{m=1}^n\alpha_m w^m}{\sum_{m=0}^n\alpha_m}$ 
\begin{equation}\label{eq:main_bound}
\E\left(\ell(x,y,w^{(n)})-\ell(x,y,w)\right)\leq \frac{U^2+G^2\sum_{m=1}^n\alpha_m^2/2}{\sum_{m=1}^n\alpha_m}
\end{equation}
\end{cor}
\begin{proof}
According to \ref{step_ineq} and by the strong convexity of $\psi(w):$
\begin{eqnarray*}
\Delta(w, w^{m+1}) & \leq &\alpha_m\la g^m, w-w^{m+1}\ra+\Delta(w, w^m)-\Delta(w^{m+1}, w^m)\\
&\leq &\la\alpha_m  g^m, w- w^{m}\ra + \la\alpha_m g^m, w^m- w^{m+1}\ra+\Delta(w, w^m)-\\
&&-\frac{1}{2}\|w^{m+1}- w^m\|^2 \\
& \leq & \la\alpha_m g^m, w-w^{m}\ra + \alpha_m\|g^m\| \|w^m- w^{m+1}\|+\Delta(w, w^m)-\\
&&-\frac{1}{2}\|w^{m+1}- w^m\|^2
\end{eqnarray*}
Summing over $m=1,\ldots,n$ gives:
\begin{equation}\label{eq:mid}
0\leq\sum_{m=0}^n\la\alpha_m g^m, w-w^{m}\ra+
\frac{1}{2}\sum_{m=1}^n\alpha_m^2\|g_m\|^2+\Delta(w, w^1)
\end{equation}
As $w^1=\underset{w\in\cW}{\argmin}\, \psi(w):$
\[\Delta(w, w^1)=\psi(w)-\psi(w^1)-\la \psi'(w^1), w^1
-w\ra\leq \psi(w)-\psi(w^1)\leq U^2,\] 
By convexity of $\ell$ and the independence of $(x_m,y_m)$ and $w^m:$
\[\sum_{m=0}^n\alpha_m (\ell(x_m,y_m,w^m)-\ell(x_m,y_m,w))\leq
\frac{1}{2}\sum_{m=1}^n\alpha_m^2\|g_m\|^2+U^2\]
$$\sum_{m=1}^n\alpha_m\left(F(w^{(n)})-F(w)\right)\leq\sum_{m=1}^n\alpha_m \big(\E\left(\ell(x,y,w^m)-\ell(x,y,w)\right)\big)\leq
\frac{1}{2}\sum_{m=1}^n{\alpha_m^2} G^2+{U^2}$$
\end{proof}

If steps are constant $\alpha_m=\alpha=\frac{\sqrt{2}U}{G\sqrt{n}}$ we get
$$F(w^{(n)})-F(w^*)\leq \frac{\sqrt{2}UG}{\sqrt{n}}$$

The dependence in the number of classes in the excess risk bound is hidden in
constants $U$ and $G$ and depends on the choice of distance-generating 
function $\psi(w)$ and the sets $\cW, \cX.$

\section{Probability inequalities}
An additional assumption on the exponential moments can be used to get probability estimate of the large deviations.
\begin{cor}
If there exists $\sigma>0:$s.t. for all $w\in\cW:$
$\E_{(x,y)}e^{d(x_i,y_i,w)/\sigma^2}\leq e,$ где $d(x_i,y_i,w) = \|g(x_i,y_i,w)-\E_{(x,y)}g(x,y,w)\|,$ then $\forall\theta>0,$ $g = \underset{w\in\cW}{\max}\|\E\ell'(x,y,w)\|_*:$
\begin{gather*}
\P\lf \E_{(x,y)}\ls\ell(x,y,w^{(n)})-\ell(x,y,w^*)\rs >\sum_{m=1}^n{\alpha_m\gamma_m} g^2+\frac{U^2}{\sum_{m=1}^n\alpha_m} +\right.\\
\left.\theta\left(\sqrt{U\sigma^2\sum_{m=1}^n\gamma_m^2}+
\sum_{m=1}^n \alpha_m\gamma_m \sigma^2\right)\rf \leq 
e^{1-\theta}+e^{-\theta^2/4}
\end{gather*}
\end{cor}
\begin{proof}
To prove the bound we use Chernoff's inequality and inequality \eqref{eq:mid}. Denote $\gamma_m=\frac{\alpha_m}{\sum_{m=1}^n\alpha_m},$
$s_m=\la \E_{(x,y)}g(x,y,w^m)-g(x_m,y_m,w^m),w^m-w^*\ra.$ Note that
$\E s_m/\sigma^2=0,$ 
$\E_{(x_m,y_m)} e^{s_m^2/(\sigma^2 2\sqrt{2}U)}\leq \E_{(x_m,y_m)} e^{d(x_m,y_m,w)/\sigma^2}\leq e^1.$ 

Consider a random variable $y:$ $\E y=0,\, \E e^{y^2}\leq e.$ 
As $e^y\leq y+e^{2y^2/3},$ for $0<\alpha^2<2/3$ by Jensen's inequality
$\E e^{\alpha y}\leq \E e^{2y^2\alpha^2/3}\leq e^{2\alpha^2/3}\leq e^{\alpha^2}.$ As
$\E e^{\alpha x}\leq e^{1+\alpha^2/4},$ $\E e^{\alpha x}\leq e^{\alpha^2}$ for $\alpha>0.$ We apply the inequality to $y^2=s_m^2/(2\sqrt{2}\sigma^2 U):$
\begin{eqnarray*}
\E\ls\exp{\left(\alpha\sum_{m=1}^n\gamma_m s_m\right)}\rs\leq 
\E\ls\exp{\left(\alpha\sum_{m=1}^{n-1}\gamma_m s_m\right)}
\E_{(x_n,y_n)}\exp{\left(\alpha\gamma_n s_n\right)}\rs\leq\\
\E\ls\exp\left({\alpha\sum_{m=1}^{n-1}\gamma_m s_m} \right)
\exp{\left(8\alpha^2\gamma_n^2U^2\sigma^4\right)}\rs\leq 
e^{2\sqrt{2}\alpha^2U\sigma^2\sum_{m=1}^n \gamma_m^2}
\end{eqnarray*}
Then  $\P\{\sum_{m=1}^n \gamma_m s_m>\theta\}\leq e^{2\sqrt{2}\alpha^2U\sigma^2\sum_{m=1}^n \gamma_m^2-\alpha\theta}$
and for $\alpha = \frac{\theta}{4\sqrt{2}U\sigma^2\sum_{m=1}^n \gamma_m^2}:$
\begin{equation}\label{eq:prob1}
\P\lf\sum_{m=1}^n \gamma_m s_m>\theta\sqrt{U\sigma^2\sum_{m=1}^n\gamma_m^2}\rf\leq e^{-\frac{\theta^2}{8\sqrt{2}}}
\end{equation}
Let us now derive an upper bound on the $\E\ls \exp{\left(\sum_{m=1}^n{\alpha_m\gamma_m}d(x_m,y_m,w^m)^2\right)}\rs:$
by the convexity of $\E e^x$ in  $x:$
\[\E\ls{\exp\frac{\sum_{m=1}^n \alpha_m\gamma_m d^2(x_m,y_m,w^m)}{\sum_{m=1}^n \alpha_m\gamma_m \sigma^2}}\rs\leq e\]
By the Chernoff's inequality:
\begin{equation}\label{eq:prob2}
\P\lf \sum_{m=1}^n \alpha_m\gamma_m d^2(x_m,y_m,w^m) >\theta\left(\sum_{m=1}^n \alpha_m\gamma_m \sigma^2\right) \rf\leq
e^{1-\theta}
\end{equation}
Use \eqref{eq:mid} to bound the $\E_{(x,y)}\ls\ell(x,y,w^{(n)})-\ell(x,y,w^*)\rs:$
\begin{eqnarray*}
\E_{(x,y)}\ls\ell(x,y,w^{(n)})-\ell(x,y,w^*)\rs &\leq &\sum_{m=1}^n\gamma_m \la \E_{(x,y)} g(x,y,w^m),w^m-w^*\ra  \\
&\leq &\sum_{m=1}^n \gamma_m s_m + \sum_{m=1}^n\gamma_m \la g(x_m,y_m,w^m),w^m-w^*\ra\\
&\leq &
\sum_{m=1}^n \gamma_m s_m+\sum_{m=1}^n{\alpha_m\gamma_m} (g^2+d(x_m,y_m,w^m)^2)+\\
&+&\frac{U^2}{\sum_{m=1}^n \alpha_m}
\end{eqnarray*}

Combining \eqref{eq:prob1} and \eqref{eq:prob2} we get:
\begin{gather*}
\P\lf \E_{(x,y)}\ls\ell(x,y,w^{(n)})-\ell(x,y,w^*)\rs >\sum_{m=1}^n{\alpha_m\gamma_m} g^2+\frac{U^2}{\sum_{m=1}^n\alpha_m} +\right.\\
\left.\theta\left(\sqrt{U\sigma^2\sum_{m=1}^n\gamma_m^2}+
\sum_{m=1}^n \alpha_m\gamma_m \sigma^2\right)\rf \leq 
e^{1-\theta}+e^{-\theta^2/4}
\end{gather*}
\end{proof}
For the constant stepsize policy $\alpha = \frac{\sqrt{2}U}{G\sqrt{n}}:$
\[\P\lf \E_{(x,y)}\ls\ell(x,y,w^{(n)})-\ell(x,y,w^*)\rs >\frac{3UG}{\sqrt{2n}}+
\theta\left(\sqrt{\frac{U\sigma^2}{n}}+
\frac{\sqrt{2}U\sigma^2}{G\sqrt{n}}\right)\rf \leq 
e^{1-\theta}+e^{-\theta^2/4}\]

{\bf Examples}
\begin{enumerate}
\item Consider 
$\cW=\{w\in \R^{d\times k}\vert \underset{i}{\max}\|w_i\|_2\leq\Omega\},$
$\cX=\{x\in \R^d\vert\|x\|_2<X\}$ and
$\psi(w)=\frac{1}{2}{\sum_{i=1}^k\|w_i\|_2^2}.$ In this case 
$\Delta(w^1, w^2)=\frac{1}{2}\sum_{i=1}^k\|w^1_i-w^2_i\|_2^2$ and 
$U^2= k\Omega^2,$  $G^2=\frac{2X^2}{\rho^2}.$
This leads to excess risk rate 
$$F(w^{(n)})-F(w^*)\leq\frac{{2}\Omega X}{\rho}\sqrt{\frac{k}{n}}$$
and for the large deviations:
\begin{eqnarray*}
\P\lf \E_{(x,y)}\ls\ell(x,y,w^{(n)})-\ell(x,y,w^*)\rs \right. & > &\frac{3X\Omega}{\rho}\sqrt{\frac{k}{n}}\\
&+&
\left.\theta\left(\sqrt{\frac{k\Omega\sigma^2}{n}}+
\frac{\Omega\sigma^2\rho}{X}\sqrt{\frac{k}{n}}\right)\rf \\
& \leq & e^{1-\theta}+e^{-\theta^2/4}
\end{eqnarray*}
\item If the margins are allowed to be different for different classes:
$\cW=\{w\in \R^{d\times k}\vert {\sum}_{i=1}^k\|w_i\|_2\leq\Omega\}$ and 
$\psi(w)$ is chosen to be strongly convex w.r.t. $\ell_1/\ell_2$ norm: $\psi(w)=\frac{e\ln k}{1+1/\ln k}\sum_{i=1}^k\|w^i\|_2^{1+1/\ln k}$ 
\cite{juditsky2011first}, which is the case of favorable geometry, then
$U^2=e\ln k\Omega,$ $G=X/\rho$ and the rate can be pushed to
$$F(w^{(n)})-F(w^*)\leq\frac{X}{\rho}\frac{\sqrt{2e\Omega \ln{(k)}}}{\sqrt{n}}$$
and
\begin{eqnarray*}
\P\lf \E_{(x,y)}\ls\ell(x,y,w^{(n)})-\ell(x,y,w^*)\rs \right. &>& \frac{3X}{\rho}
\sqrt{\frac{e\ln k\Omega}{2n}}\\
&+&\left.\theta\left(\sqrt{\frac{\sqrt{e\ln k\Omega}\sigma^2}{n}}+
\frac{\sqrt{2e\ln k\Omega}\sigma^2\rho}{X\sqrt{n}}\right)\rf\\
& \leq &
e^{1-\theta}+e^{-\theta^2/4}
\end{eqnarray*}
\end{enumerate}

\section{Class probability}
We further focus on the Euclidean setup 
$\cW=\{w\in \R^{d\times k}\vert \underset{i}{\max}\|w_i\|_*\leq\Omega\},$
$\cX=\{x\in \R^d\,\vert\,\|x\|<X\}.$ As the classes can be unequally probable, we turn to $\cF_c=\lf f(x,y,w)=c_y\la x,w_y\ra\,\vert\,w\in\cW\rf$ with a fixed 
$c\in R^k:\, \|c\|_{\infty}=1.$ We also choose the norm on $\cW$ to be 
$\|w\|_b=\sqrt{\sum_{i=1}^k b_y\|w_y\|^2}$ with $b\in\R^k,\, b_i\geq 0.$ For a strongly convex w.r.t. $\|\cdot\|_*$ d.-g. function $\psi(w_i),$ we set $\hat{\psi}(w)=\frac{1}{2}\sum_{i=1}^kb_i\psi(w_i),$ which is a d.-g. f. for $\|\cdot\|_b$ and the corresponding Bregman divergence is 
$\Delta_b(w^1,w^2)=\frac{1}{2}\sum_{i=1}^n b_i\Delta(w^1_i,w^2_i),$ where 
$\Delta(w_i^1,w_i^2)$ is Bregman divergence of $\psi(w_i).$

%

Let $p(y)=\P\{y_i=y\},$ the upper bound for each step is then:
\begin{eqnarray*}
\E\Delta(w, w^{m+1}) & \leq &\alpha_m \big( F(w)-F(w^m)\big) +
\E\Delta(w, w^m)+ \alpha_m\E \la g^m, w^m- w^{m+1}\ra\\
&-&\E\Delta(w^{m+1}, w^m)=\alpha_m \big( F(w)-F(w^m)\big) +
\E\Delta(w, w^m)\\
&+&\frac{\alpha_m}{\rho}\sum_{y=1}^k p(y)\left(
\E\ls c_y\lv\la x_m, w_y^m-w_y^{m+1}\ra \rv\;\,\bigg\vert y_m=y\rs\right. \\
&+&\left. \sum_{y'=1}^k
\E\ls c_{y'} \lv\la x_m, w_{y'}^m-w_{y'}^{m+1}\ra\rv
[ y'=\underset{y'\neq y_m,y'\in\cY}{\argmax}\la x,w_{y'}\ra] \,\,\bigg\vert y_m=y\rs\right)\\
&-&\sum_{y=1}^k\frac{1}{2}\E\ls\sum_{i=1}^k b_i
\|w_i^{m+1}- w_i^m\|^2\,
\bigg\vert\,  y_m=y\rs p(y)\\
&\leq &  \alpha_m \big( F(w)-F(w^m)\big) +\E\Delta(w, w^m)+ 
\frac{\alpha_m^2 X^2}{2\rho^2} \sum_{y=1}^k
p(y)\left(\frac{c_y^2}{b_y}+\underset{y'\in\cY}{\max}{\frac{c_{y'}^2}{b_{y'}}}
 \right)
\end{eqnarray*}
To bound the last term notice that 
$y'$ is a determined function of $(x, y)$ and $w:$ 
$$p'(y, y')=\E_{(x,w|y)}p(y)
\ls y'=\underset{y''\in\cY\setminus\{y\}}{\argmax\,}
(\alpha(y'')\la x,w_{y''} \ra)\rs,$$
where $\E_{(x,w|y)}$ denotes the conditional expectation of $x$ and 
$w$ with respect to $y$

$$A=\sum_{y, y'=1}^k\left(\frac{p'(y,y')}{\gamma(y)}+
\frac{\alpha(y')^2 p'(y,y')}{\alpha(y)^2\gamma(y')}\right) =
\sum_{y=1}^k\frac{p(y)}{\gamma(y)}+$$
$$\sum_{y=1}^k\left(\frac{p(y)}{\alpha(y)^2}\E_{(x,w|y)}\sum_{y'=1}^k 
\ls y'=\underset{y''\in\cY\setminus\{y\}}{\argmax\,}
(\alpha(y'')\la x,w_{y''} \ra)\rs
\frac{\alpha(y')^2}{\gamma(y')} \right)
\leq \sum_{y=1}^k\left(\frac{p(y)}{\gamma(y)}+\frac{p(y)}
{\alpha(y)^2}\max_{y'\in \cY}\frac{\alpha(y')^2}{\gamma(y')} \right)$$ 

Setting $\alpha(y)^2=\gamma(y)=\sqrt{p(y)},$ 
$B=\sum_{y=1}^k \sqrt{p(y)}$ and summing over
$m=1, n$ we get:
$$F(w^{(n)})-F(w^*)\leq \frac{B\Omega^2+
\frac{2X^2B}{\rho^2}\sum_{m=1}^n\alpha_m^2}
{2\sum_{m=1}^n\alpha_m}$$

If $\alpha_m=\frac{\Omega\rho}{X\sqrt{2n}}:$
$$F(w^{(n)})-F(w^*)\leq \frac{\Omega X\sqrt{2}}{\rho\sqrt{n}}\sum_{y=1}^k \sqrt{p(y)}
$$

$\sum_{i=1}^k\sqrt{p(i)}\leq \sqrt{k},$ so the result is not worse than the previous 
one, but in case of highly unbalanced classes, e. g. power log law
$p(i)=Ci^{-\beta}, 
i=\overline{1, k}, \beta>2$
$$\sum_{i=1}^k \sqrt{p(i)}=\frac{\sum_{i=1}^ki^{-\beta/2}}{\sqrt{\sum_{i=1}^k i^{-\beta}}}=O(1)$$
%
%
As long as the exact prior probabilities are usually not known exactly,
we will assume that they are known up to some constant and the
estimate is independent of the sample $(x_i, y_i)_{i=1}^n$: 
$p(y)\leq(1+\eps)\hat{p}(y)$

Again, setting $\alpha(y)^2=\gamma(y)=\sqrt{\hat{p}(y)}$ we get the estimate
$$A\leq 2\sum_{y=1}^k \frac{p(y)}{\gamma(y)}\leq
2(1+\eps)\sum_{y=1}^k \sqrt{\hat{p}(y)}$$
$$F(w^{(n)})-F(w^*)\leq \frac{\Omega X\sqrt{2(1+\eps)}}{\rho\sqrt{n}}\sum_{y=1}^k \sqrt{\hat{p}(y)}
$$

The work is supported in part by the grant of the President of Russia for young PhD
 MK-9662.2016.9 and by the RFBR grant 15-07-09121a.
\bibliographystyle{chicago}
\bibliography{MD.bib}


%

\end{document}